\documentclass[11pt]{amsart}
\usepackage{color}
\usepackage{graphicx}
\usepackage{amsmath}
\usepackage{amssymb}
\usepackage{xfrac}
\usepackage{bm}
\usepackage[foot]{amsaddr}
\DeclareMathAccent{\wtilde}{\mathord}{largesymbols}{"65}
\newtheorem{theorem}{Theorem}[section]
\newtheorem{proposition}{Proposition}[section]
\newtheorem{cor}[proposition]{Corollary}
\newtheorem{remark}{Remark}[section]
\newtheorem{as}{Assumption}[section]
\newcommand{\Arg}{\mbox{\rm Arg}\ }
\newcommand{\Arctan}{\mbox{Arctan}\ }
\newcommand{\Arccot}{\mbox{Arccot}\ }

\begin{document}
\title[Surface shear waves in a half-plane]{Surface shear waves in a half-plane with depth-variant structure
}


\author{Andrey Sarychev         \and
        Alexander Shuvalov \and Marco Spadini 
}


\address{A. Sarychev,   
              DiMaI, Universit\`a di Firenze,  
              v. delle Pandette 9, Firenze, 50127 Italy; asarychev@unifi.it} 
           \address{ A. Shuvalov,   
             Universit\'e de Bordeaux,  UMR 5295, 
 Talence 33405 France;   alexander.shuvalov@u-bordeaux.fr} 
\address{M. Spadini,   
  DiMaI, Universit\`a di Firenze, 
  via S.Marta 3,   Firenze 50139 Italy; marco.spadini@unifi.it} 

\maketitle

\begin{abstract}
We consider the propagation of surface shear waves in a half-plane, whose shear
modulus $\mu(y)$ and  density $\rho(y)$ depend continuously on the depth coordinate $y$.
 The problem amounts to studying the parametric Sturm-Liouville equation on a half-line
with frequency $\omega$ and wave number $k$ as the parameters.
The Neumann (traction-free)  boundary condition and the requirement  of decay at infinity  are  imposed.
 The condition of solvability of the boundary value problem  determines the dispersion spectrum   $\omega(k)$ for  the corresponding surface wave.
 We establish  the  criteria for non-existence of surface waves  and
 for the existence of $N(k)$ surface wave solutions,  with $N(k) \to \infty$  as  $k \to \infty$.
The most intriguing result is a possibility of  the  existence of infinite number   of solutions, $N(k)=\infty$,  for any given  $k$.
These three options are conditioned by the properties of  $\mu(y)$ and  $\rho(y)$.

\end{abstract}

\section{Introduction}
\label{intro}
We consider the 2D wave equation
\begin{equation}\label{wehat}
   \rho \hat u_{tt}-\nabla \left(M \nabla \hat u\right)=0
\end{equation}
in a half-plane $\{(x,y): \ y>0\}$.
One imposes the Neumann boundary condition
\begin{equation}\label{bchat}
\hat u'_y|_{y=0}=0.
\end{equation}
We seek  the  solutions,  which decay at infinity:
 \begin{equation}\label{dinfhat}
   \lim_{y \to +\infty} \hat u =0.
 \end{equation}
We make an assumption  of  $M, \rho$  depending  only on $y$ and $M$ being  a scalar matrix $M=\mu(y)\mbox{Id}$.

 In the physical context, this is a problem of the existence of  surface shear  waves in functionally graded semi-infinite media with a traction-free boundary. Surface acoustic waves find numerous applications in various fields extending from seismology to microelectronics.
Their localization near the surface (decay into the depth) makes them extremely advantageous in non-destructive material testing for detection of surface and subsurface defects (surface wave sensors). Small wavelength of surface waves enables their application in filters and transducers used in modern miniature devices \cite{BGKP}.
Functionally graded materials may be of natural origin (e.g. bones), they may occur due to material aging, or they may be specially manufactured to realize desired combination of physical properties \cite{MA}.

 Under the adopted assumptions equation \eqref{wehat}
 reads as
 \begin{equation}\label{y_dep}
   \rho(y) \hat u_{tt}=\mu (y)  \hat u_{xx}+\partial_y(\mu(y)\partial_y \hat u).
 \end{equation}
 We will seek solutions  of the form
 \[\hat u(x,y)=u(y)e^{i(kx-\omega t)}.\]
 Substituting $\hat u(x,y)$ into \eqref{y_dep} and  cancelling   $e^{i(kx-\omega t)}$,  one
 gets   for $u(y)$ the  equation
 \[\rho(y)u(y)(-\omega^2)=\mu(y)u(y)(-k^2)+\partial_y(\mu(y)\partial_y u(y)).\]
  We denote the (total) derivative $\partial_y$ by $'$
arriving at the second-order linear differential equation
 \begin{equation}\label{ueq}
  \left(\mu(y)u'(y)\right)'+(\omega^2 \rho(y)-k^2 \mu(y))u=0.
 \end{equation}
The boundary conditions \eqref{bchat} and \eqref{dinfhat} formulated for $u(y)$ become
\begin{eqnarray}
  u'(0)=0, \label{bc}\\
  \lim_{y \to +\infty} u(y) =0. \label{dinf}
\end{eqnarray}
We assume both functions $\rho(y)$ and $\mu(y)$ to be continuous and  positive on $[0,+\infty)$; further assumptions are introduced in Sections \ref{sec:1},\ref{section_parametric}.

  It is known that for generic $ \omega , k$ there are no solutions of \eqref{ueq},  which satisfy  both boundary
  conditions \eqref{bc} and \eqref{dinf}. For many
  bi-parametric problems  the set of admissible $\omega , k$ is known to be a union of a number of {\it eigencurves} (\cite[Ch.6]{AM}) in $\omega k$-plane,
  called in the physical context {\it dispersion curves}.
Our goal is to characterize the pairs $(\omega , k)$, for which the solutions of the boundary value problems \eqref{ueq}-\eqref{bc}-\eqref{dinf} exist.

  The situation is elementary, when $\rho(y), \mu(y)$ are constants,
and is relatively uncomplicated, when  $\rho(y),\mu(y)$ become constants on an interval   $[y_s , +\infty)$.
 In Section \ref{sec_res} we briefly consider the latter \textit{homogeneous substrate case} as a particular case of our general treatment.
There has been  a number of studies, which either treat the problem asymptotically for high $\omega , k$ or assume that  $\rho(y), \mu(y)$ are periodic  \cite{AB}, \cite{Ti}, \cite{XFK}, \cite{Shu1}, \cite{Shu2}.
We address the   case, where no bounds for  $\omega , k$   are imposed  and  neither periodicity nor (piecewise) constancy for  $\rho(y), \mu(y)$  is   assumed.

The paper has the following structure. Section~\ref{sec:1} contains the auxiliary results.  In Section \ref{section_parametric} we formulate the  corresponding
 parametric Sturm-Liouville problem on a half-line and introduce the assumptions for the material coefficients $\rho(y), \mu(y)$.
  Section \ref{sec_res} contains the formulations of  the main results, which are   the  criteria for non-existence of surface waves (Theorem \ref{neg_monot})  and
 for the existence of $N(k)$ surface wave solutions,  with $N(k) \to \infty$  as  $k \to \infty$ (Theorem \ref{conv_ginf}).
The most intriguing result is a possibility of  the  existence of infinite number   of solutions, $N(k)=\infty$,  for any given  $k$ (Theorem \ref{div_ginf}).
These three options are conditioned by the properties of  $\mu(y)$ and  $\rho(y)$.
        Section \ref{sec_proofs} contains the  proofs of the above Theorems.


\section{Second-order linear ordinary differential equation on a half-line: auxiliary results}
\label{sec:1}


\subsection{Second-order linear equation}

Equation \eqref{ueq} is a particular type of the second-order linear  differential equation
 \begin{equation}\label{SL}
  \left(\mu(y)u'(y)\right)'+\gamma(y)u=0
 \end{equation}
defined on a half line $[0,+\infty)$.

\begin{as}\label{as_mu}
We assume from now on that the function $\mu(s) \geq \underline{\mu}>0$ on  $[0,+\infty)$, is  continuous on  $[0,+\infty)$ and
admits a  finite limit $\lim_{s \to \infty}\mu(s)=\mu_\infty>0.\ \qed$
\end{as}

The following substitution of the independent variable
\begin{equation}\label{subs_tau}
 \tau (y)=\int_0^y \left(\mu(s)\right)^{-1}ds
\end{equation}
is invertible ($\tau (y)$ is strictly growing) and satisfies the relation:
$\frac{d}{d\tau}=\mu \frac{d}{dy}$.

By Assumption \ref{as_mu},  the functions $\mu(s), \left(\mu(s)\right)^{-1}$ are both bounded on $[0,+\infty)$ and  therefore
the function $\tau(y)$ and its inverse $y(\tau)$ are Lipschitzian. Besides $\int_0^{+\infty}\left(\mu(s)\right)^{-1}ds=\infty$, i.e.
$\tau(y)$ is Lipschitzian homeomorphism of $[0,+\infty)$ onto $[0,+\infty)$.

This substitution transforms  \eqref{SL}  into  the standard form
\begin{equation}\label{1SL}
 \frac{d^2 \bar u}{d\tau^2}+\bar \gamma (\tau) \bar u(\tau)=0,
\end{equation}
where $\bar u(\tau)=u(y(\tau))$  and  $\bar \gamma (\tau)= \mu (y(\tau))\gamma (y(\tau))$. 

Another  form of \eqref{SL} is its representation as  a system of first-order differential equations for the variables $u(y), \ w(y)=\mu(y)u'(y)$:

\begin{eqnarray}\label{sys}
  u'(y)=\frac{w(y)}{\mu(y)}, \ w'(y)=-\gamma (y)u(y) ,
  \end{eqnarray}
or in the matrix form for $Z =\left(
                          \begin{array}{c}
                            w \\
                            u \\
                          \end{array}
                        \right)$:
\begin{equation}\label{heq_mat}
  \  Z '(y) =\frac{d Z}{dy} =C(y)Z(y), \ \ C(y)=\left(
                                       \begin{array}{cc}
                                         0 & -\gamma (y)  \\
                                         \left(\mu(y)\right)^{-1} & 0 \\
                                       \end{array}
                                     \right).
\end{equation}

Performing substitution \eqref{subs_tau},  we transform \eqref{heq_mat} into the system for the function $\bar Z(\tau)=Z(y(\tau))$
\begin{equation}\label{heq_tau}
  \frac{d\bar Z}{d\tau} =\bar C(\tau)\bar Z(\tau), \ \ \bar C(\tau)=\left(
                                       \begin{array}{cc}
                                         0 & -\bar \gamma (\tau)  \\
                                         1 & 0 \\
                                       \end{array}
                                     \right).
\end{equation}

We concentrate   for a moment   on the asymptotic properties of the solutions of \eqref{SL}, \eqref{1SL}, \eqref{heq_mat}, \eqref{heq_tau} at infinity.

\subsection{Asymptotic properties of solutions for  $y \to +\infty$}

The matrix of the coefficients $C(y)$  of  the system \eqref{heq_mat}  for each $y$ is  traceless, hence,  by the Liouville formula,  the Wronskian of a fundamental system of solutions
is constant in $y$. This  precludes a possibility of having two independent solutions, which would both tend to zero at infinity.

Important characteristics of the asymptotics of the system at infinity are determined by the limit of the coefficient matrix for $y \to +\infty$  (whenever it exists):
\[C_\infty=\lim_{y \to +\infty}C(y)=\left(
                                       \begin{array}{cc}
                                         0 & -\gamma_\infty   \\
                                         (\mu_\infty)^{-1} & 0 \\
                                       \end{array}
                                     \right),\]
where $\mu_\infty=\lim_{y \to +\infty}\mu(y), \ \gamma_\infty=\lim_{y \to +\infty}\gamma(y)$.

Whenever $\det C_\infty= \gamma_\infty (\mu_\infty)^{-1}> 0$, or,  equivalently, $\gamma_\infty > 0$,  the eigenvalues of $C_\infty$ are purely  imaginary and  one can conclude (see Proposition \ref{ellip} below)  the non-existence of a solution of system \eqref{heq_mat} with $\lim_{y \to +\infty} u(y)=0$.

{\it If on the contrary $\det C_\infty <0$, then the eigenvalues of $C_\infty$ are real numbers of opposite signs and the existence of a  solution of \eqref{heq_mat} with $\lim_{y \to \infty} u(y)=0$ is guaranteed under some additional conditions on the functions}
$\mu(y), \gamma (y)$.

Note that $\det \bar C_\infty=\mu_\infty^2 \det C_\infty$ and therefore a similar conclusion holds for the solutions of  system
\eqref{heq_tau}.

Later on   we  use a number of   results  which    follow the quasi-classical or WKB-appro\-xi\-mation paradigm (\cite[Ch.2]{Fed}).
We formulate the results  for   equations \eqref{SL} or \eqref{1SL}.





Let us introduce  linear space $\mathcal{G}$ of the coefficients $\gamma(y)$ of equations \eqref{SL} as
a space of functions   $\gamma(y)=\gamma_\infty+ \beta (y)$, with
$\gamma_\infty$ being a constant and $\beta (y)$ a continuous function  on $[0,+\infty)$ such that:
\begin{eqnarray}
  \lim_{y \to +\infty}\beta (y)=0, \label{betainf}  \\
  \int_{0}^{+\infty}|\beta(y)|dy < \infty . \label{beta_L1}
 \end{eqnarray}
Evidently $\lim_{y \to \infty}\gamma(y)=\gamma_\infty$.

Introduce in $\mathcal{G}$  the norm
\begin{equation}\label{norma}
  \|  \gamma(\cdot)\|_{01}=|\gamma_\infty|+\|\beta(\cdot)\|_{C^0}+\|\beta(\cdot)\|_{L_1}.
\end{equation}
For each $y_0 \in [0,+\infty)$ we define a subset $\mathcal{G}^-(y_0) \subset \mathcal{G}$ (respectively   $\mathcal{G}^+(y_0) \subset \mathcal{G}$),  consisting  of the functions $\gamma(y)=\gamma_\infty+ \beta (y)$,  for which $\gamma_\infty <0$ (respectively $>0$) and $\gamma_\infty+ \beta (y)<0$  (respectively $>0$) on $[y_0,+\infty)$. Both $\mathcal{G}^-(y_0)$ and $\mathcal{G}^+(y_0) $
are open subsets of $\mathcal{G}$ in the above introduced norm.
It is easy to verify that substitution \eqref{subs_tau} transforms the space $\mathcal{G}$ into itself and  the sets  $\mathcal{G}^-(y_0), \mathcal{G}^+(y_0) $
 into $\mathcal{G}^-(\tau(y_0)), \mathcal{G}^+(\tau(y_0)) $,  correspondingly.



 The first classical result regards the so called non-elliptic case for equation \eqref{1SL},
 where the coefficient  $\bar \gamma(\cdot) \in  \mathcal{G}^-(\tau_0)$.

\begin{proposition} (see \cite[\S 6.12]{Bel}]).
\label{asymp}
  Consider   the equation
  \begin{equation}\label{Har_SL}
    u''(\tau) +\bar \gamma (\tau)u=u''(\tau) +\left(-\lambda^2+\beta(\tau)\right)u=0, \ \lambda >0.
  \end{equation}
Assume   $\beta(\tau)$  to be continuous and to satisfy \eqref{betainf}.
Then for  equation \eqref{Har_SL} there exist $\tau_0 \geq 0$,  constants $c_1,c_2,d_1, c'_1,c'_2,d'_1$ and two  solutions $u_{\lambda}(\tau), u_{-\lambda}(\tau)$ such that
$\forall \tau \geq \tau_0$:
\begin{eqnarray}
  c'_2exp\left[\lambda \tau-d'_1\!\int_{\tau_0}^{\tau}\!\!\!\left|\beta (\theta)\right|d\theta\right] \leq u_\lambda (\tau)\leq \nonumber \\  \leq c'_1 exp\left[\lambda \tau+d'_1\!\int_{\tau_0}^{\tau}\!\!\!\left|\beta (\theta)\right|d\theta\right],
  \label{u+} \\
  c_2exp\left[-\lambda \tau-d_1\!\int_{\tau_0}^{\tau}\!\!\!\left|\beta (\theta)\right|d\theta\right] \leq u_{-\lambda} (\tau) \leq \nonumber \\  \leq c_1exp\left[-\lambda \tau+d_1\!\int_{\tau_0}^{\tau}\!\!\!\left|\beta (\theta)\right|d\theta\right].
  \label{u-}
\end{eqnarray}

\end{proposition}

\begin{cor}(see \cite[\S XI.9]{Har}).
  Assume the assumptions of Proposition \ref{asymp} to hold and $\beta(\cdot)$ to satisfy \eqref{beta_L1}.
Then the solutions $u_\lambda, u_{-\lambda}$ satisfy
\[ u_{\lambda} \sim \frac{u_{\lambda}'}{\lambda} \sim e^{\lambda \tau}, \  u_{-\lambda} \sim -\frac{u_{-\lambda}'}{\lambda} \sim e^{-\lambda \tau}\]
as $\tau \to +\infty$.
\end{cor}

\begin{cor}
\label{bliz}
For each
$\tilde \gamma (\cdot)$ sufficiently close to  $\bar \gamma (\cdot)$
in the norm \eqref{norma} the equation
\[u''(\tau)+\tilde \gamma (\tau)u(\tau)=0\]
has a decaying solution.
\end{cor}

Next we pass on to the   elliptic case  (see \cite[\S XI.8]{Har}; Corollary 8.1), where the coefficient  $\bar \gamma(\cdot) \in  \mathcal{G}^+(y_0)$.

\begin{proposition}\label{ellip}
Consider   the equation
  \begin{equation}\label{Har_SL+}
    u''(\tau) +\bar \gamma (\tau)u=u''(\tau) +\left(\lambda^2+\beta(\tau)\right)u=0,  \  \lambda >0
  \end{equation}
with  $\bar \gamma(\cdot) \in  \mathcal{G}^+(y_0)$.  Then for any real $a,b$ there is a unique solution
of equation \eqref{Har_SL+}  with the asymptotics
\begin{eqnarray}\label{u_osc}
 u(\tau)=\left(a+o(1)\right)\cos \lambda \tau + \left(b+o(1)\right)\sin \lambda \tau ,   \\
 u'(\tau)=\left(-\lambda a+o(1)\right)\sin \lambda \tau + \left(\lambda b+o(1)\right)\cos \lambda \tau  , \nonumber
\end{eqnarray}
as $\tau \to +\infty$.
\end{proposition}

\subsection{Pr\"ufer's coordinates}\label{prufer}
We consider  Pr\"ufer's coordinates (see \cite{Har,AM}):
\begin{equation}\label{Pruf1}
r=(u^2+\mu^2u'^2)^{\frac{1}{2}}=(u^2+w^2)^{\frac{1}{2}}, \ \varphi =\mbox{Arctg}\ \frac{u}{w},
\end{equation}
where again $w=\mu u'$.
For the vector function $Z =\left(
                          \begin{array}{c}
                            w \\
                            u \\
                          \end{array}
                        \right)$
we denote $\varphi$  by $\Arg Z$ (the choice of a continuous branch is done in a standard way).
In   coordinates \eqref{Pruf1}  system \eqref{SL} takes the form:
\begin{equation}\label{eq_Pruf1}
  r'=\left(\mu^{-1}(y)-\gamma(y)\right)r \sin \varphi \cos \varphi , \
  \varphi '=\gamma (y) \sin^2 \varphi +\mu^{-1}(y) \cos^2 \varphi ;
\end{equation}
note that the second equation is decoupled from the first one.

We list some  facts concerning the evolution of $\Arg Z(y)$. Recall that   $\mu(y)$ in  equation \eqref{SL} meets Assumption~\ref{as_mu}.
\begin{proposition}
\label{corollary_Prufer}
 \begin{enumerate}
  \item[i)] If $\gamma(y)\geq 0$ (respectively $\gamma (y) >0$) on an interval,
  then for a solution $Z(y)$ of \eqref{sys}  Pr\"ufer's angle variable $\varphi=\Arg Z$ is non-decreasing (increasing) on the interval.
  \item[ii)] If $\gamma(y)<0 $ on an interval $I$, then the first  and the third
   quadrants -- $\Arg Z \in (0,\sfrac{\pi}{2})$  and  $\Arg Z \in (\pi , \sfrac{3\pi}{2})$  --  are invariant       for  system \eqref{sys} on $I$.
  \item[iii)] For any $\gamma(y)$ there is a kind of weakened monotonicity for $\Arg Z$: if $\Arg Z(\tilde y) > m\pi$, then
$\Arg Z(y) > m\pi$ for any $y>\tilde y$.
\end{enumerate}
\end{proposition}

Property i) follows from \eqref{eq_Pruf1}. So does  property ii), since,  according to \eqref{eq_Pruf1}, $\varphi'(\pi m) >0$ and  $\varphi'(\sfrac{\pi}{2}+\pi m)<0$  for negative
$\gamma$.  Property iii) follows from  the fact that in \eqref{eq_Pruf1} $\varphi '(m\pi)=\mu^{-1}(m\pi)>0$.


\subsection{Oscillatory equations}
\label{subsec_oscil}

Second-order linear differential equation is {\it oscillatory} (\cite[\S XI.5 ]{Har}) on $[0,+\infty)$ when its every solution has infinite number of zeros  on $[0,+\infty)$,  or equivalently the set of zeros of any solution has no upper limit, or equivalently for every solution  its Prufer's coordinate $\Arg Z$  (see the previous Subsection) satisfies
\[\limsup_{y \to +\infty} \Arg Z(y)=+\infty .   \]
An obvious example of oscillatory equation is \eqref{Har_SL+}, when the assumptions of Proposition \ref{ellip} are met.

We are interested in  conditions, under which the same  equation is oscillatory for vanishing  $\lambda$.
 We formulate the result   (see \cite[\S XI.5]{Har},
\cite[Ch.2,\S 6]{Fed})   for equation   \eqref{1SL}.

\begin{proposition}
Let $\bar \gamma(\cdot)$ in \eqref{1SL} be continuous  of bounded variation on every interval $[0,T]$,  $\bar \gamma(\tau)>0$ on some interval $[\tau_0,+\infty)$, and
\begin{eqnarray}
  \int_{\tau_0}^{+\infty}\left(\bar \gamma(\tau)\right)^{1/2}d\tau=+\infty ,  \label{inf_sq_root}  \\
  \int_{\tau_0}^T \left(\bar \gamma(\tau)\right)^{-1}|d \gamma (\tau)| = o\left(\int_{\tau_0}^{T}\left(\bar \gamma(\tau)\right)^{1/2}d\tau \right), \ \mbox{as} \
  T \to +\infty . \label{cond_o_int}
\end{eqnarray}
Then equation \eqref{1SL} is oscillatory.
\end{proposition}



\subsection{Hamiltonian form}\label{ham_form}
One can rewrite the system
\eqref{heq_mat} in the Hamiltonian form
\begin{equation}\label{heq}
  u'=\frac{\partial H}{\partial w}=\frac{w}{\mu(y)}, \ w'=-\frac{\partial H}{\partial u}=-\gamma (y)u
\end{equation}
with the Hamiltonian
 \[ H= \frac{1}{2}\left(\frac{w^2}{\mu(y)}+\gamma(y)u^2\right).\]
We denote by $\overrightarrow h$ the (Hamiltonian) vector field at the right-hand side of \eqref{heq}.

For  Pr\"ufer's angle $\varphi= \Arctan \left(\frac{u}{w}\right)$ there holds
\[  \varphi '= \frac{-w'u+wu'}{u^2+w^2}=\frac{\gamma u^2+\sfrac{w^2}{\mu}}{u^2+w^2}= \frac{2H}{u^2+w^2}.\]
The last equation is equivalent to  the differential equations \eqref{eq_Pruf1} for Pr\"ufer's  coordinate $\varphi$.

\begin{remark}
\label{der_uw}
 A simple but relevant (see \cite{Arn}) computation is provided by derivation of $u(y)w(y)$ along the trajectories of Hamiltonian system \eqref{heq}:
\begin{equation}\label{dup}
  \frac{d}{dy}\left(uw\right)=\partial_{\overrightarrow h}(uw)=\left(\partial_{\overrightarrow h}u\right)w+u\left(\partial_{\overrightarrow h}w\right)=-\gamma u^2+\frac{w^2}{\mu},
\end{equation}
wherefrom it follows, among other things, that $uw$ is nondecreasing (respectively increasing) on the intervals where $\gamma(y) \leq 0$ (respectively $\gamma(y) < 0$).
\end{remark}

 Proposition \ref{corollary_Prufer} and Remark \ref{der_uw}  allow us to  arrive at a conclusion on  qualitative behaviour of solutions on an interval,  where $\gamma (\tau) <0$ in \eqref{Har_SL}.

According to Proposition \ref{asymp},  there is a decaying solution, along which (according    to Remark \ref{der_uw}) $uw$ grows.
Hence the solution approaches the origin either in the second or in the fourth
quadrant, where $uw<0$.

\begin{proposition}\label{pi2pi} Let $\bar \gamma (\tau)$ meet the assumptions of Proposition \ref{asymp} and $\bar \gamma (\tau) < 0$ for  $\tau \in [\tau_0, +\infty)$. Then  the decaying solutions $\pm u(\tau)$  of \eqref{Har_SL} correspond to the solutions
$\pm Z(\tau)$ of \eqref{sys} with $\Arg Z (\tau) \in [\sfrac{\pi}{2},\pi]$,
$\Arg (-Z) (\tau) \in [\sfrac{3\pi}{2},2\pi]$  for $\tau \in [\tau_0, +\infty)$.
\end{proposition}

Other solutions, which  start in the same quadrants, escape to
either the first or the third quadrant, which, according to Proposition   \ref{corollary_Prufer}, are invariant for \eqref{Har_SL} whenever  $\gamma (\tau) <0$.
According to Remark \ref{der_uw},  the product $uw$ (positive in these quadrants)   grows along the respective trajectories, which  tend to infinity.

\subsection{Sturmian properties of trajectories}
We provide few  results from the Sturm theory. First result is classical (\cite{Fed},\cite{AM}, \cite[Ch. X,XI]{Har})
 and follows directly from the second equation \eqref{eq_Pruf1}.



\begin{proposition}[comparison result]\label{SL_comparison}
  Consider a pair of  second-order equations
  \begin{equation}\label{SL_pair}
    \left(\mu(y)u'(y)\right)'+\gamma(y)u=0,  \
     \left(\mu(y)u'(y)\right)'+\tilde \gamma(y)u(y)=0,
  \end{equation}
where $\mu (y)$ meets Assumption \ref{as_mu} and
\[\tilde \gamma(y) \geq \gamma (y), \ \forall y \in [y_0,+\infty).\]
If for $y_1 \geq y_0$ and a pair of vector solutions $Z =\left(
                          \begin{array}{c}
                            w \\
                            u \\
                          \end{array}
                        \right), \ \tilde Z =\left(
                          \begin{array}{c}
                            \tilde w \\
                            \tilde u \\
                          \end{array}
                        \right)$ of the first and the second equations
\eqref{SL_pair}
$$\Arg \tilde Z(y_1)= \Arg Z(y_1),$$
then
 \[\forall y \geq y_1: \ \Arg \tilde Z(y) \geq  \Arg Z(y)\]
and
\begin{equation}\label{-fifi1}
 \forall y \in [y_0, y_1]: \ \Arg \tilde Z(y) \leq  \Arg Z(y).
\end{equation}
\end{proposition}

We provide  analogue of the comparison result 
(in particular,  of relation \eqref{-fifi1}) for the decaying solutions of  \eqref{SL_pair},  when $y_1=+\infty$.
  We were not able  to trace it  in the literature and  provide a  (short) proof.

\begin{proposition}{\bf (comparison result for decaying solutions on a half-line)}\label{ic_decaying_solution}
  Consider  the pair of second-order  equations \eqref{SL_pair}
with  the coefficient $\mu (y)$ meeting Assumption \ref{as_mu} and with   $\gamma (y), \tilde \gamma (y)$ belonging  to  $\mathcal{G}^-(y_0)$. Let
\begin{equation}\label{gaga0}
0 > \tilde \gamma(y) \geq \gamma (y), \ \forall y \in [y_0,+\infty).
\end{equation}
If $Z, \tilde Z$  are the decaying solutions of  equations \eqref{SL_pair},
then 
\begin{equation}\label{in_cond_decay}
\Arg \tilde Z (y)  \leq \Arg  Z(y), \ \forall y \geq y_0.
\end{equation}
\end{proposition}

\begin{proof} Without lack of generality we may assume $\mu(y) \equiv 1$; otherwise we  perform substitution \eqref{subs_tau} of the independent variable, which preserves relation \eqref{gaga0} for the coefficients.

By \eqref{gaga0} and \eqref{dup},  the functions $u  w$ and $\tilde u \tilde w$ are increasing  on $[y_0, +\infty)$.
As long as the limits of these functions at $+\infty$  are null, we conclude that $(u w)(y)<0,(\tilde u  \tilde w)(y)<0$ on $[y_0, +\infty)$  and then without lack of generality  we may assume that $u(y),  \tilde u(y)$ are  positive, while $w(y),  \tilde w(y)$ are  negative on  $[y_0, +\infty)$.

Denote $\tilde \gamma (y)-\gamma (y)$ by $\Delta \gamma (y)$  and represent the second one of  equations \eqref{SL_pair} as
\begin{equation}\label{perturb_form}
 \tilde u''+ \gamma (y)\tilde u=-\Delta \gamma (y) \tilde u ;
\end{equation}
$\Delta \gamma (y)>0$ by \eqref{gaga0}.

 Applying the  integral form of the Lagrange identity (or Green's formula, see \cite[\S XI.2]{Har}) to the respective vector  solutions $Z =\left(
                          \begin{array}{c}
                            w \\
                            u \\
                          \end{array}
                        \right), \ \tilde Z =\left(
                          \begin{array}{c}
                            \tilde w \\
                            \tilde u \\
                          \end{array}
                        \right)$ of equations \eqref{SL_pair}, of which the second one is  written  as \eqref{perturb_form},   we conclude:
   \[\forall y \geq y_0: \ \left.\left(u \tilde w-w \tilde u\right)\right|_y^{+\infty}=\int_y^{+\infty}
   -\Delta \gamma (s) \tilde u(s) u(s)ds <0.\]
 Given that $\left(u \tilde w-w \tilde u\right)$   vanishes at $+\infty$,  we obtain:
\begin{equation}\label{Lag_identity}
\forall y \geq y_0: \
-u(y) \tilde w(y)+w(y) \tilde u(y)=\int_y^{+\infty}
   -\Delta \gamma (s) \tilde u(s) u(s)ds <0.
      \end{equation}
Dividing the inequality in \eqref{Lag_identity}  by the positive value
$w(y) \tilde w (y)$,  we get
\[\forall y \geq y_0: \ \frac{\tilde u(y)}{\tilde w(y)} \leq \frac{u(y)}{w(y)},  \]
wherefrom \eqref{in_cond_decay}  follows.
\end{proof}

   We establish the continuous dependence of decaying solutions on the coefficient $\gamma(\cdot)$ in $\|\cdot\|_{01}$-norm.

\begin{proposition}{\bf (continuous dependence of decaying solutions on the right-hand side)}\label{decaying_continuity}
Consider equations \eqref{SL_pair}.
Let $\gamma (\cdot)=-\lambda^2+\beta(\cdot) \in \mathcal{G}^-_{y_0}$  for some $y_0 \in [0,+\infty)$.
 Then for any $\tilde \gamma (\cdot)=-\tilde \lambda^2+\tilde \beta (\cdot) $   sufficiently close to $\gamma(\cdot)$ in   $\|\cdot\|_{01}$-norm: 

 i)  both   equations \eqref{SL_pair} possess the decaying vector solutions $Z(\cdot), \tilde Z(\cdot)$ with  $\Arg Z$, $\Arg \tilde Z \in [\sfrac{\pi}{2}, \pi]$;

    ii)   for each  $y \in [y_0, +\infty)$
    \[\left|\Arg \tilde Z(y) - \Arg Z(y)\right| \to 0, \ \text{as} \   \|\tilde \gamma (\cdot)- \gamma (\cdot)\|_{01} \to 0.\]
 \end{proposition}

 \begin{proof}Again we may proceed assuming $\mu(y) \equiv 1$.

i)  Any $\tilde \gamma(\cdot)$ sufficiently close to $\gamma(\cdot)$ in $\|\cdot\|_{01}$-norm  belongs to  $\mathcal{G}^-_{y_0}$,  which  is open with respect to the norm.   The existence of the decaying solutions $Z(y), \tilde Z(y)$ follows from Corollary \ref{bliz}. Since both $\gamma$ and $\tilde \gamma$ are negative   on $[y_0,+\infty)$,  we conclude by Proposition \ref{pi2pi} that
$\Arg Z(y)$ and $ \Arg \tilde Z(y)$ lie in $[\sfrac{\pi}{2}, \pi]$ for $y \in [y_0,+\infty)$.

This implies that for $s \in [y_0,+\infty), \ w(s), \tilde w(s) $ are negative,  while   $u(s), \tilde u(\tau)$ are positive and  by \eqref{sys} decrease.

ii)   Recall that   $\Delta \gamma (\cdot)=\tilde \gamma (\cdot)-\gamma(\cdot)$.  Invoking the equality in \eqref{Lag_identity} and dividing it
by  $-u(y) \tilde u (y)$,   we get
\begin{equation}\label{Delta_Z}
 \frac{\tilde w(y)}{\tilde u(y)} -\frac{w(y)}{u(y)}  = \int_y^{+\infty}
   \Delta \gamma (s) \frac{\tilde u(s)}{\tilde u(y)}\frac{u(s)}{u(y)}ds=\int_y^{+\infty}
   \Delta \gamma (s)\nu(s)\tilde \nu (s)d\tau ,
\end{equation}
where  $\nu(s)=\frac{\tilde u(s)}{\tilde u(y)}, \ \tilde \nu (s)=\frac{u(s)}{u(y)}$ are the  solutions of the  first and second equation \eqref{SL_pair},  which are
normalized by
 the condition: $\nu(y)=\tilde \nu(y)=1$.

 By the aforesaid $\nu(s), \ \tilde \nu (s)$ decrease; hence
 \begin{equation}\label{nu1}
   \nu(s) \leq 1, \tilde \nu(s) \leq 1, \ \text{for}\   s \geq y .
 \end{equation}
  According to Proposition \ref{asymp},  there  exist $c_1,d_1>0, s_0 >y$ such that
  \begin{equation}\label{est_nu}
    \nu(s) \leq c_1 exp\left(-\lambda s + d_1 \int_{s_0}^{s}|\beta (\sigma)|d\sigma\right), \ \forall s > s_0.
  \end{equation}
   From the proof of the Proposition (see \cite[\S 6.12, \S 2.6]{Bel}) it follows  that one can choose
 in  \eqref{est_nu} any $c_1>1$, a sufficiently large $d_1$ and then choose $s_0$ such that  $d_1 \sup_{s \geq s_0}|\beta(s)| < \lambda$.
The same holds for the second one of  equations \eqref{SL_pair}.

For each  $\tilde \gamma$ from a small neighborhood of  $\gamma$ in $\|\cdot \|_{01} $-norm,
$\tilde \lambda$ and $\lambda$  as well as   $\sup_{\tau \geq \tau_0}|\beta(\tau)|$ and $\sup_{\tau \geq \tau_0}|\tilde \beta(\tau)|$ are close.
Thus  one can choose common $c_1,d_1,\tau_0$ for all the equations with  the coefficient $\tilde \gamma$  from the neighborhood.
Besides  there is a common upper bound $B$ for the corresponding norms
$\|\tilde \beta(\cdot)\|_{L_1}$.
Then by \eqref{Delta_Z},\eqref{nu1} and \eqref{est_nu}
\[ \left|\int_y^{+\infty}
   \Delta \gamma (s)\nu(s)\tilde \nu (s)ds \right|  \leq \int_{y}^{s_0}\left|\Delta \gamma (s)\right|d s +
c^2_1 e^{ 2d_1B}\int_{s_0}^\infty  e^{-\lambda s}\left|\Delta \gamma (s)\right|d s \]
with the right-hand side tending  to $0$ as $\left\|\Delta \gamma (s)\right\|_{01} \to 0$.

Note that $\Arg Z=\Arccot \frac{w(y)}{u(y)}, \ \Arg \tilde Z=\Arccot \frac{\tilde w(y)}{\tilde u(y)}$ and since the function $z \mapsto \Arccot z$ is Lipschitzian with constant $1$:
 \begin{equation*}
   \left|\Arg Z(y) - \Arg \tilde  Z(y)\right|=\left|\Arccot \frac{w(y)}{u(y)}-\Arccot  \frac{\tilde w(y)}{\tilde u(y)}\right|  \leq \left|\frac{w(y)}{u(y)}-\frac{\tilde w(y)}{\tilde u(y)}\right|
 \end{equation*}
and the left-hand side tends to $0$ as $\left\|\Delta \gamma (\tau)\right\|_{01} \to 0$.

 \end{proof}

\section{Existence of surface waves and parametric Sturm-Liouville problem}
\label{section_parametric}

We come back to equation \eqref{ueq} and  simplify the notations  putting
$\Omega=\omega^2, \  K=k^2, \ A=(K,\Omega)$,
\begin{equation}\label{def_gamma}
  \gamma_A(y)= \Omega \rho(y)-K \mu(y),
\end{equation}
thus arriving at  the   equation
\begin{equation}\label{SL_par}
  (\mu(y)u')'(y)+\underbrace{\left(\Omega \rho(y)-K \mu(y)\right)}_{ \gamma_A(y)}u(y)=0
\end{equation}
\emph{with the parameter} $A$.

Performing the substitution of the independent variable the way it is done in  \eqref{subs_tau}, we get
the equation:
\begin{equation}\label{1SL_par}
\frac{d^2 \bar u}{d\tau^2}+\underbrace{\bar \mu (\tau) \left(\Omega \bar \rho(\tau)-K \bar \mu(\tau)\right)}_{\bar \gamma_A(\tau)} \bar u(\tau)=0,
\end{equation}
where
\begin{equation}\label{subs_all}
  \bar \rho(\tau)=\rho(y(\tau)), \ \bar \mu(\tau ) =\mu(y(\tau)), \ \bar u(\tau)=u(y(\tau)).
\end{equation}

In equations \eqref{SL_par} and \eqref{1SL_par} the dependence of the coefficients on the parameters $\Omega , K$ is linear;  the functions $\rho(y) , \ \mu(y), \ \bar \rho(\tau), \ \bar \mu(\tau )$ are positive. Note that $\bar \mu(0)=\mu(0), \bar \rho(0)=\rho(0)$ and
\[ \bar \mu(+\infty)=\mu(+\infty), \ \bar \rho(+\infty)=\rho(+\infty), \ \bar \gamma_A(+\infty)=\Omega \mu_\infty \rho_\infty -K\mu_\infty^2.\]

We know from the previous Section that if equation \eqref{1SL_par}  meets the assumptions of  Proposition~\ref{asymp}, then  it has a   solution,
which satisfies  the boundary condition at infinity \eqref{dinf}.
We are interested, though,  in the solutions,  which satisfy at the same time   the boundary condition \eqref{bc}, and it is not possible for  generic combinations of  $\bar \rho(y), \bar \mu(y), \Omega , K$, which  enter  \eqref{1SL_par} via the coefficient  $\bar \gamma_A(\cdot)$.
In other words we  get  parametric Sturm-Liouville  problem on a half-line for  equation \eqref{1SL_par} (or \eqref{SL_par}) with the boundary conditions  \eqref{bc}-\eqref{dinf}.

Let us  introduce the vector-function  $a(y)=(\rho(y),\mu(y))$, which characterizes our medium,  and  formulate the assumptions for the medium in terms of $a(y)$.

\begin{as}[Lipschitz continuity]
  \label{lip_cont}
  The function  $a(y)=\left(\rho (y),\mu(y) \right)$ is  Lipschitz continuous  on $[0,+\infty)$.
  There exists a    finite   limit
  \[\lim_{y \to + \infty}a (y)=a_\infty,  a_\infty=\left(\rho_\infty , \mu_\infty\right), \ \rho_\infty>0 , \ \mu_\infty>0.\]
\end{as}

\begin{as}[integral boundedness]
  \label{bound_L1}
  The function  \[\hat a(y)=(\hat \rho(y), \hat \mu(y))=a(y)-a(+\infty)=(\rho(y)-\rho_\infty , \mu(y)-\mu_\infty)\] is integrable on $[0,+\infty)$: $\int_0^\infty \left|\hat a(y)\right|dy < \infty $.
\end{as}




We now introduce  monotonicity assumptions  formulated in terms of  polar coordinates representation for $a(y)$.
Let
\begin{eqnarray*}
  |a(y)|=\left((\rho (y))^2+(\mu(y))^2 \right)^{1/2}, \ \Arg a(y)=\Arctan \frac{\mu(y)}{\rho(y)},  \\
a(y)=\left(\rho (y),\mu(y) \right)=|a(y)|\left(\cos \Arg a(y), \sin \Arg a(y)  \right).
\end{eqnarray*}
The values $a(y)$ have both positive coordinates; hence the values of  $\Arg a(y)$  lie  in $[0, \sfrac{\pi}{2}]$.
As long as $a_\infty \neq 0$,  $\Arg a_\infty$ is properly defined.




\begin{as}[monotonicity at infinity]
\label{monot}
      There exists an interval  $\bar I=(\bar y , +\infty)$  such that   either: i) $\Arg a(y) < \Arg a_\infty$ on $\bar I$ - positive monotonicity at infinity, or ii) $\Arg a(y) > \Arg a_\infty$ on $\bar I$ - negative monotonicity at infinity.
\end{as}

    Examples of  the curves $a(y)=\left(\rho(y), \mu(y)\right)$ are drawn in  Figure 1   together  with the vector $A=(K,\Omega)$.
The curves (1) and (2) are negatively monotonous at infinity, while the curve (3) is positively monotonous at infinity.

\begin{figure}
  \centering
\includegraphics[width=280pt]{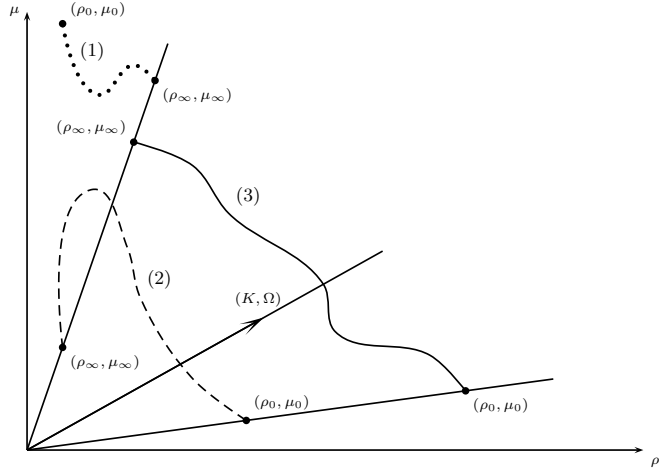}
\caption{Curves  parameterized by $y$, which  exhibit different types of monotonicity at infinity}
\label{fig:1}
\end{figure}

 Assume  the vector of parameters $A=(K,\Omega)$  to belong to (the positive quadrant of)  the oriented plane, in  which the  curve $y \mapsto a(y), \ y \in [0,+\infty]$ is contained.
We define  $\Arg A = \Arctan \frac{\Omega}{K}$.
The following remarks are important.

\begin{remark}\label{rem_property_gammA} Fix an admissible $A$.
If
$\Arg a(y) < \Arg  A$ (respectively $\Arg a(y) > \Arg  A$)  for some  $y$,    then   $\gamma_A(y)$  defined by \eqref{def_gamma}  is positive (respectively negative).  $\qed$
 \end{remark}

\begin{remark}
  Under   Assumptions \ref{lip_cont}  and \ref{bound_L1},   for any admissible $A$ and $\gamma_A(\cdot)$ defined by \eqref{def_gamma}  there holds:
  \begin{enumerate}
    \item  $\gamma_A  (y) - \gamma_A  (+\infty) \stackrel{y \to +\infty}{\longrightarrow} 0$ ;
\item  $\int_0^\infty |\gamma_A(y)-\gamma_A(+\infty)| dy < \infty $;
\item  if $\gamma_A (y)$ admits positive values, then so does $\gamma_{A'}(y)$ with any  $A'$  such that  $\Arg A'$  is sufficiently close to $\Arg A$. $\qed$
     \end{enumerate}
\end{remark}

We wish to check what occurs with   Assumptions \ref{lip_cont}-\ref{bound_L1}-\ref{monot}   after  substitution  \eqref{subs_tau}.

\begin{proposition}\label{assump_subs} Let Assumption \ref{as_mu} hold and let equation \eqref{SL_par}  meet   Assumptions \ref{lip_cont}, \ref{bound_L1}, \ref{monot} for any admissible $A$ .   Then equation \eqref{1SL_par}
meets the same  Assumptions.
\end{proposition}

\begin{proof}
 By Assumption \ref{as_mu},  $\tau(y)$  defined by \eqref{subs_tau}   is  Lip\-schitzian homeomorphism
 of $[0,+\infty)$ onto itself. Hence  the functions $\bar \mu , \bar \rho$  defined by \eqref{subs_all}   are bounded, Lipschitzian,  with finite limits at infinity, i.e.  Assumption \ref{lip_cont} is valid for them.

 Under substitution  \eqref{subs_tau},   the vector-function  $a(y)=(\rho(y),\mu(y)$ is transformed into
  $\bar a(\tau)=\bar \mu(\tau)\left(\bar \rho(\tau), \bar \mu(\tau)\right)$. Hence $\Arg a(y)=\Arg \bar a (\tau(y))$ and all the monotonicity properties  listed in Assumption \ref{monot}  are maintained.

Regarding  Assumption \ref{bound_L1} we perform  substitution \eqref{subs_tau}  and obtain:
 \begin{eqnarray*}
   \int_{0}^{+\infty}\left|\bar \gamma_A(\tau)-\bar \gamma_A(+\infty)\right|d\tau =\int_{0}^{+\infty}\frac{\left|\mu(y) \gamma_A(y)-
   \mu(+\infty)\gamma_A(+\infty)\right|}{\mu(y)}dy= \\
   =\int_{0}^{+\infty}\left|\left(\gamma_A(y) - \gamma_A(+\infty)\right) + \gamma_A(+\infty) \left(\mu(y)-\mu(+\infty)\right)(\mu(y))^{-1}\right|dy < \infty,
 \end{eqnarray*}
 since  $\left(\mu(y)\right)^{-1}$ is bounded on $[0,+\infty)$.
\end{proof}

For the {\it limit case}, in which $A_\infty=(K_\infty , \Omega_\infty)=\beta a_{\infty}, \ \beta >0$,   or in other words $\Arg A_\infty=\Arg a_\infty$,
we get
\[\gamma_{A_\infty}(y)=\Omega_\infty \rho(y) - K_\infty \mu(y)=
\beta(\mu_\infty(\rho_\infty+\hat \rho(y))-\rho_\infty (\mu_\infty+\hat \mu(y))=\beta \hat  \gamma_\infty (y),   \]
where
\begin{equation}\label{ginfin}
   \hat \gamma_\infty (y)=\mu_\infty \hat \rho(y)- \rho_\infty \hat \mu (y).
\end{equation}

\begin{remark}
\begin{enumerate}
\item Under Assumption \ref{lip_cont},  for each $A=(K,\Omega)$ with $\Arg A  < \Arg a_\infty$ there exists an interval, $[y_-, +\infty)$, on which $\gamma_A(y)<0$.

\item Under  Assumption \ref{monot}i) (respectively \ref{monot}ii)),
    there is an interval  $[\bar y, +\infty)$, on which   $\gamma_\infty (y)$  is positive  (respectively  negative).
\end{enumerate}
\end{remark}

\section{Results}\label{sec_res}

     Key information for our treatment is provided by the   {\it limit-case equation},  which corresponds to the vectors of parameters $A_\infty=(K_\infty,\Omega_\infty)=\beta a_\infty, \ \beta >0$.
      For such choice of parameters
       equation \eqref{SL_par} takes the form
\begin{equation}\label{SL_limit}
  (\mu(y)u')'+\beta \hat \gamma_\infty (y)u=0
\end{equation}
with $\hat \gamma_\infty(y)$  as in \eqref{ginfin}.



We formulate here  main  results of the paper;  the proofs are provided in the next Section.
Our first result establishes non-existence of solutions under a kind of global negative monotonicity of $a(y)$ at infinity.

\begin{theorem}
\label{neg_monot}
Let  assumptions \ref{lip_cont}-\ref{bound_L1} hold
and
\begin{equation}\label{a>inf}
\forall y \in [0,+\infty): \ \Arg a(y) \geq \Arg a_\infty.
\end{equation}
  Then there are
no admissible values of parameters $K, \Omega$,  for which solutions
of \eqref{SL_par}-\eqref{bc}-\eqref{dinf} exist.
\end{theorem}

\begin{remark}\label{rem_A<a-inf}
The assumptions of the theorem are met by  curve (1) in Fig. 1.
The proof of the result is  based on the following fact:
for any $A=(K,\Omega)$ such that   $\Arg A \leq \min_{y \in [0,+\infty)}\Arg a(y)$, or the same,
\[\Omega \leq  K \min_{y \in [0,+\infty)}\frac{\mu(y)}{\rho(y)}=\frac{\mu(\check{y})}{\rho(\check{y})}=\frac{\check{\mu}}{\check{\rho}},\]  solutions of \eqref{SL_par}-\eqref{bc}-\eqref{dinf} do not exist. $\qed$
\end{remark}

 If \eqref{a>inf} does not hold, then  one can guarantee existence of solutions at least for sufficiently large $K,\Omega$.

\begin{theorem}\label{conv_ginf}
Let assumptions \ref{lip_cont}-\ref{bound_L1}-\ref{monot}
 hold and in addition $\Arg a(y) < \Arg a_\infty $ for $y \in I$ - a non-null sub-interval of $[0,+\infty)$.
Then  for each $N >0$   $\exists K_N$ such that  $\forall K > K_N$ there are at least $N$ values  $\Omega_j \in \left(\frac{\check{\mu}}{\check{\rho}}K, \frac{\mu_\infty}{\rho_\infty}K\right)$, $j=1, \ldots , N$,
such that for each $(K,\Omega_j)$  the solution of
\eqref{SL_par}-\eqref{bc}-\eqref{dinf} exists.
\end{theorem}

\begin{remark}\label{rem_pos_mon}
The curves (2) and (3) in Fig. 1 meet  assumptions of the Theorem. $\qed$
\end{remark}

Finally there is a case,  in which
for each $K>0$ one finds a numerable set of
$\Omega_j \in \left(\frac{\check{\mu}}{\check{\rho}}K, \frac{\mu_\infty}{\rho_\infty}K\right)$  such that the solution exists for $(K,\Omega_j)$.  It happens when the limit-case equation \eqref{SL_limit} is  oscillatory  (see Subsection \ref{subsec_oscil}).


\begin{theorem}\label{div_ginf}
Let  assumptions \ref{lip_cont}-\ref{bound_L1}-\ref{monot}i)
 hold
and    the limit-case equation \eqref{SL_limit} be oscillatory\footnote{We assume  \eqref{inf_sq_root} and \eqref{cond_o_int} to hold}.

Then for each $\bar K > 0$ there exists  a numerable set of
$\Omega_m \in \left(\frac{\check{\mu}}{\check{\rho}}\bar K, \frac{\mu_\infty}{\rho_\infty}\bar K\right)$, $m=1, \ldots ,$
such that:

i) for $A_m=(\bar K,\Omega_m)$  the  solution  of
\eqref{SL_par}-\eqref{bc}-\eqref{dinf} exists;

ii)  $\Omega_m$ increase with $m$ and accumulate (only) to $\bar \Omega = \frac{\mu_\infty}{\rho_\infty}\bar K$;

iii) for the vector   solutions $Z(y; A_m)$ there holds  \[\Arg Z(y; A_m) \in \left[\left(m-\sfrac{1}{2}\right)\pi, m \pi\right]  \ \mbox{for $y$ sufficiently large.}\]
\end{theorem}

\begin{remark}
Assumptions \ref{lip_cont}-\ref{bound_L1}-\ref{monot}i)
 hold for curve (3) in Fig. 1, but the oscillatory property for the limit-case equation can not be concluded from the curve only, since it also depends on its parametrization. $\qed$
\end{remark}

\begin{figure}
  \centering
\includegraphics[width=200pt]{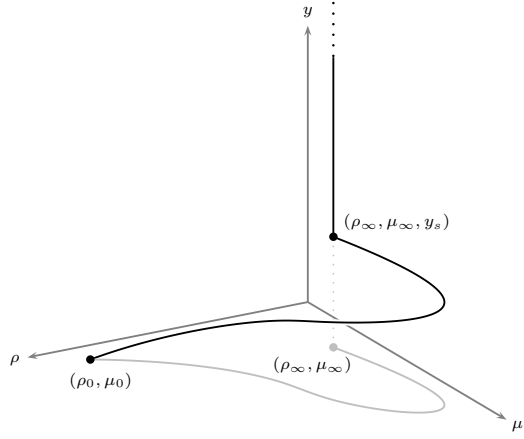}
\caption{The functions $\rho(y)$ and $\mu(y)$ in the  homogeneous substrate example become constant when $y\geq y_s$, as illustrated by the curve in black. Note that its projection (in gray) on the $(\rho,\mu)$-plane is a curve, which  exhibits negative monotonicity at infinity.}
\label{fig:2}
\end{figure}

\subsection{Homogeneous substrate example} This is a particular case, in which the properties of the medium become depth-independent starting from some depth.
For the  model under discussion this means existence of $y_s$ such that $\mu(y)$ and $\rho(y)$ are constant on the interval $[y_s, +\infty)$: $\mu(y) \equiv \mu_s, \ \rho(y) \equiv \rho_s$ on $[y_s, +\infty)$ (see Fig. 2).

We denote  $a_s=(\rho_s, \mu_s)$.
 Then  $a_\infty=\lim_{y \to \infty}a(y)=a_s$ and $\hat a(y)=a(y)-a_\infty$ vanishes on $[y_s, +\infty)$.

If $\forall y \in [0, +\infty): \ \Arg a(y) \geq \Arg a_s$ or, the same
\[\forall y \in [0, +\infty): \ \frac{\mu(y)}{\rho(y)} \geq \frac{\mu_s}{\rho_s},\]
then we are under assumptions of Theorem \ref{neg_monot} and solutions
of \eqref{SL_par}-\eqref{bc}-\eqref{dinf} do not exist.

If $\frac{\mu(y)}{\rho(y)} < \frac{\mu_s}{\rho_s}$ on some non-null subinterval of
$[0,+\infty)$, then we fall under assumptions of Theorem \ref{conv_ginf} and hence its claim holds.

\section{Proofs}\label{sec_proofs}
Since substitution \eqref{subs_tau} transforms parametric equation \eqref{SL_par} into its standard form \eqref{1SL_par}  and  Assumptions  \ref{lip_cont},\ref{bound_L1},\ref{monot} are maintained under  \eqref{subs_tau}, we may take,   without  loss  of  generality,   $\mu(y)
 \equiv 1$ in \eqref{SL_par}.

The {\it proof of Theorem \ref{neg_monot}} is easy. Pick some $A=(K,\Omega)$.
There are two options: $\Arg a_\infty \leq  \Arg  A$ or $\Arg a_\infty > \Arg  A$.

In the first case,  by monotonicity and continuity assumptions,  the coefficient
$\gamma_A(y)$ in equation \eqref{SL_par} is non-negative  on some interval $[y_0,+\infty)$.
 Then,  by Proposition~\ref{ellip}, there exists a fundamental system of solutions of the form
 \eqref{u_osc}  and none of  the  solutions of \eqref{SL_par} tend to the origin  as $y \to +\infty$.

If $\Arg A < \Arg a_\infty \leq \Arg a(y)  \ \forall y \in [0,+\infty)$, then $\gamma_A(y) < 0$ on $[0,+\infty)$.  By \eqref{dup},  for  a solution $Z(y,A)=\left(
                                        \begin{array}{c}
                                          w \\
                                          u \\
                                        \end{array}
                                      \right)$
there holds $\frac{d}{dy}\left(u(y)w(y)\right)>0$.
 This enters in contradiction with the boundary conditions \eqref{bc}-\eqref{dinf}, according to which $u(0)w(0)=0$ and $\lim_{y \to +\infty}\left(u(y)w(y)\right)=0$.

The latter reasoning  also validates the claim of Remark \ref{rem_A<a-inf}.

\vspace{2mm}

{\it Proof of Theorem~\ref{div_ginf}}.
We start with a {\it sketch of the proof}.

Take a vector of parameters $A_\infty=(\bar K, \bar \Omega)$ collinear to $a_\infty=(\rho_\infty, \mu_\infty)$ and consider its perturbation
$A_{\infty , s}=(\bar K, \bar \Omega-s)$. It is immediate to see that for each $s>0$
 equation \eqref{SL_par} with  $A=A_{\infty , s}$
 and the coefficient
 \begin{equation}\label{gam_Am}
 \gamma_{A_{\infty, s}}(y)=\gamma_{A_\infty}(y)- s \rho(y)=- s \rho_\infty +(\bar \Omega - s)\hat \rho(y)-\bar K \hat \mu(y)
\end{equation}
 meets the assumptions of Proposition \ref{asymp},  and hence
  the equation
\begin{equation}\label{eq_gamma_eps}
  u''+\gamma_{A_{\infty,\bar s}}(y)u=0
\end{equation}
 possesses a decaying solution
$Z^+(y,A_{\infty, s})$.

Simultaneously we consider  the solutions $Z^0(y,A_{\infty , s})$
of the same equation with  the boundary condition \eqref{bc}.
The goal is to detect the values  $s>0$,  for which the  solutions $Z^0(y,A_{\infty , s})$ and $Z^+(y,A_{\infty, s})$ {\it meet} at some intermediate point $\bar y \in [0,+\infty)$, i.e admit at $\bar y$ the same value (mod $\pi$). In such a case  they (or their opposites) can be  concatenated into solutions of \eqref{SL_par}-\eqref{bc}-\eqref{dinf}.
The possibility of such {\it meeting} follows from Propositions \ref{SL_comparison} and \ref{ic_decaying_solution}, according to which  for a sufficiently large  intermediate point $\bar y \in [0,+\infty)$   the vectors $ Z^0(\bar y,A_{\infty , s})$  and   $ Z^+(\bar y,A_{\infty , s})$ rotate in opposite directions  as $s$ grows
 from some $\bar s>0$.

One can assume (increasing $\bar y$ if necessary) that  $\forall s \geq \bar s$
one has $\gamma_{A_{\infty, s}}(y)<0$ on $(\bar y, +\infty)$ and
$\Arg Z^+(\bar y,A_{\infty , s}) \in (\sfrac{\pi}{2},\pi)$.
On the other hand, for small $s>0, \ \Arg Z^0(\bar y,A_{\infty , s})$ is close to  $\Arg Z^0(\bar y,A_{\infty})$, which, due to the oscillation property of the limit-case equation, tends to $+\infty$ as $\bar y \to +\infty$. Therefore for each natural $m$ one can find (again increasing $\bar y$ when necessary) small $\bar s>0$  such that $\Arg Z^0(\bar y,A_{\infty , \bar s}) > \pi m$.  As $s$ will grow from $\bar s$ to $\bar \Omega$, $\Arg Z^0(\bar y,A_{\infty , \bar s})$ will decrease from the value greater than $ \pi m$  to the value less than $\pi$ and during this  evolution  it  becomes equal  (mod $\pi$)  to $\Arg Z^+(\bar y,A_{\infty , s})$ for $m$ distinct values of  $s$.

Now we provide the detailed proofs of the statements i)-iii) of the Theorem.

i) Pick $\bar K>0$ and take $\bar \Omega = \frac{\mu_\infty}{\rho_\infty}\bar K$, so that $A_\infty=\left(\bar K,\bar \Omega\right)$ is collinear with $a_\infty$.
 Consider  the limit-case  equation  \eqref{SL_limit} with  the parameter $A_\infty$ and choose  the solution $Z^0(\cdot;A_\infty)$,   which satisfies  the boundary  condition \eqref{bc}.
As long as  equation \eqref{SL_limit} is oscillatory,  $\Arg Z^0(y;A_\infty) $ tends to infinity as $y \to +\infty$. Hence,   {\em for each  natural} $m$   $\exists y_m \in [0,+\infty)$ such that
$\Arg Z^0(y_m;A_\infty) > \pi m$.

By the continuity of the trajectories of \eqref{SL_par} with respect to the parameter $A$,  one can find $\bar s >0$ such that for any $ s \in (0, \bar s]$ and for
$A_{\infty,s} =(\bar K, \bar \Omega -s)$   there holds
 $\Arg Z^0(y_m;A_{\infty,s}) > \pi m $.

        For the function $\gamma_{A_{\infty,\bar s}}(y)$   defined by \eqref{gam_Am}  one can find $\bar y \geq y_m$ such that $\gamma_{A_{\infty,\bar s}}(y)<0$  on $[\bar y,+\infty)$.
It follows from Remark \ref{corollary_Prufer}iii)  that
$\Arg Z^0(\bar y;A_{\infty,\bar s})$ $> \pi m$.

 The second-order equation \eqref{eq_gamma_eps} for  $s=\bar s$
 meets the assumptions of Proposition \ref{asymp} and hence has  the  decaying solution $Z^+(y;A_{\infty,\bar s})$.
By Proposition~\ref{pi2pi},   there holds:
\[\forall y \geq \bar y: \ \Arg Z^+(y;A_{\infty,\bar s}) \in
\left(\sfrac{\pi}{2}, \pi  \right) \ (\mbox{mod} \ \pi).\]

 Letting $s$ grow from $ \bar s$ towards $\bar \Omega$,  we note that the values of $\gamma_{A_{\infty,s}}(y)=\gamma_{A_\infty}(y)-s \rho(y)$ on $[0,+\infty)$ diminish; in particular,  $\gamma_{A_{\infty,s}}(y)<0$ for  $y \in [\bar y,+\infty)$ for all $s \geq \bar s$.
According to Proposition \ref{SL_comparison},  the function  \linebreak $s \to \Arg Z^0(\bar y;A_{\infty,s})$ decreases monotonously from the value $\Arg Z^0(\bar y;A_{\infty,\bar s})>\pi m$ to  the  value  $\Arg Z^0(\bar y;A_{\infty,\bar \Omega})\in \left(0, \pi \right)$.

Consider now the decaying solutions $Z^+(y;A_{\infty,s})$. Proposition \ref{ic_decaying_solution} implies that  for chosen $\bar y$   $\Arg Z^+(\bar y;A_{\infty,s})$
grows with the growth of $s$, remaining $(\mbox{mod} \ \pi)$ in the interval $(\sfrac{\pi}{2}, \pi)$.
During the evolution there occur (at least)  $m$ values of  $s_j, \ j=1, \ldots , m,$ for which
\[\Arg Z^+(\bar y;A_{\infty,s_j})=\Arg Z^0(\bar y;A_{\infty,s_j})-\pi n \ (n \ \mbox{\rm  - integer}) . \]
Then the concatenations
\begin{equation}\label{concat}
  Z(y;A_{\infty,s_j})=\left\{
                             \begin{array}{ll}
                              & Z^0( y;A_{\infty,s_j}), \  \hbox{$y \leq \bar y$,} \\
                               &(-1)^nZ^+( y;A_{\infty,s_j}), \  \hbox{$y \geq \bar y$,}
                             \end{array}
                           \right.
\end{equation}
are the decaying solutions of the corresponding equations
\[u''+\left((\bar \Omega -s_j)\rho(y)-\bar K\mu(y)\right)u=0, \]
and \eqref{concat}  satisfies  the boundary condition
\eqref{bc}-\eqref{dinf}.

ii) Let $\tilde \Omega \in (0,\bar \Omega)$ be a limit point of
$\Omega_n=\bar \Omega -s_n, \ n=1, \ldots $. Then $\tilde \Omega=\bar \Omega -\tilde s< \bar \Omega$.

Consider $\gamma_{A_{\infty,\tilde s}}$. There exists $\tilde y$, such that
$\gamma_{A_{\infty,\tilde s}} < 0$ on $[\tilde y, +\infty)$.
Pick the decaying solution
$Z^+(y;A_{\infty, \tilde  s})$.  According to the aforesaid $\forall y \in [\tilde y, +\infty)$:
$\Arg Z^+(y;A_{\infty, \tilde  s}) \in (\sfrac{\pi}{2}, \pi) \ (\mbox{mod} \ \pi)$.

Consider the solution $Z^0(y;A_{\infty,\tilde s})$, which meets the initial condition \eqref{bc}.
If $\Arg Z^+(\tilde y;A_{\infty,\tilde s}) \neq \Arg Z^0(\tilde y;A_{\infty,\tilde s}) \ (\mbox{mod} \ \pi)$, then the inequality  holds for values of $s$ close to $\tilde s$,   and in particular  for all $s_n$,  but finite number of them, and this results in a contradiction.

 Let $\Arg Z^0(\tilde y;A_{\infty,\tilde s})-\Arg Z^+(\tilde y;A_{\infty,\tilde s})=\pi m $. Since   $\Arg Z^0(\tilde y;A_{\infty,s})-\Arg Z^+(\tilde y;A_{\infty, s})$ decreases with the growth of $s$, one concludes: \[\Arg Z^+(\tilde y;A_{\infty, s}) \neq \Arg Z^0(\tilde y;A_{\infty, s}) \ (\mbox{mod} \ \pi)\]
for all $s \neq \tilde s$ from a  sufficiently small neighborhood of $\tilde s$ and hence for all $s_n$ but a finite number of them, which leads us to the same contradiction.

 iii) By the construction  provided in i), for each natural $m$, there exist $A_m=(\bar K, \Omega_m)$ and the  decaying solution $Z(y,A_m)$ of \eqref{SL_par}-\eqref{bc}-\eqref{dinf}, which converges to the origin in such a way that $\Arg Z(y,A_m) \in [\pi (m-\sfrac{1}{2}), \pi m ]$ for sufficiently large $y$.

 To prove its uniqueness,  we   assume on the contrary that there exists another $A'=(\bar K, \Omega')$ and a decaying solution of \eqref{SL_par}-\eqref{bc}-\eqref{dinf} such that for  $y \in [y_0 , +\infty)$
 $\gamma_{A_m}(y)<0, \gamma_{A'}(y)<0$  and
  both $\Arg Z(y,A'), \Arg Z(y,A_m)$  belong to $[\pi (m-\sfrac{1}{2}), \pi m ]$ for  $y \in [y_0 , +\infty)$.

  Let  for example $\Omega' > \Omega_m$. Then $\gamma_{A_m}(y) < \gamma_{A'}(y)$ and
  hence  \linebreak $\Arg Z(y_0,A_m) <  \Arg Z(y_0,A')$. This enters in contradiction with the result of Proposition \ref{ic_decaying_solution}.

{\it Proof of Theorem~\ref{conv_ginf}}. Let us
pick $\bar K>0$ and take $\bar \Omega=\bar \Omega = \frac{\mu_\infty}{\rho_\infty}\bar K$, so that $A_\infty=\left(\bar K,\bar \Omega\right)$ is collinear with $a_\infty$. By assumptions of the Theorem,   the function $\gamma_{A_\infty}(y)$ admits positive values on some non-null subinterval $(\underline c,c) \subset  [0,+\infty)$. The same holds true for  $\gamma_{\beta  A_\infty}$ with $\beta  A_\infty=(\beta \bar K, \beta \bar \Omega), \ \beta >0$.

  Our proof can  be accomplished along the lines of the proof of Theo\-rem~\ref{div_ginf} if  one proves that for any  $N$  there exists $\beta_N>0$, for which the solution $Z^0(y,\beta_N  A_\infty)$ with  initial condition \eqref{bc} satisfies  $\Arg Z^0(c,\beta_N  A_\infty) > \pi N$.

Consider the equation
  \[u''(y)+\gamma_{\beta  A_\infty}(y)u=u''(y)+\beta \gamma_{A_\infty}(y)u=0\]
on the interval $[0,c]$.
It is known (\cite[\S A.3, \S A.5]{AM}) that  the number of zeros of the  solution $u(y,\gamma_{\beta A_\infty}(\cdot))$, or, the same, the increment of  Pr\"ufer's angle
\[\Arg Z(y,\gamma_{\beta A_\infty}(\cdot))-\Arg Z(0,\gamma_{\beta A_\infty}(\cdot))\] grows    as
\begin{equation}\label{Nk}
  \pi^{-1}\beta^{\sfrac{1}{2}} \int_0^y \left(\max(\gamma_{A_\infty}(\eta),0)\right)^{1/2}d\eta   +O(\beta^{\sfrac{1}{3}})
\end{equation}
as $\beta \to +\infty$.
Hence choosing sufficiently large $\beta >0$,  we can get a solution   $Z^0(y;\beta A_\infty)$ with  the boundary  condition \eqref{bc} and such that
$\Arg Z^0(c;\beta  A_\infty) > N \pi$.
It follows from Proposition \ref{corollary_Prufer}  that
$Z^0(y;\beta A_\infty) > N \pi , \ \forall y >c$.

The rest of the proof follows the proof of Theorem \ref{div_ginf}.  One can also conclude from
\eqref{Nk} that  $N(k) \sim  k$  as $k \to \infty$,  where  $N(k)$ is the number of surface wave solutions with a  given wave number $k=K^{1/2}$.


\end{document}